\documentclass[10pt]{article}
\usepackage{graphicx}

\usepackage{amsmath, amsthm, amssymb}
\usepackage{setspace}
\usepackage{graphicx}
\def\R{{\mathbb R}}
\def\N{{\mathbb N}}
\newtheorem {definition} {Definition}
  \newtheorem {remark}     {Remark}
  
  \newtheorem {theorem}    {Theorem}
  
  \newtheorem {corollary}  {Corollary}
  \newtheorem {proposition}{Proposition}
  
  \newtheorem {problem} {Problem}

\begin{document}
\title{A note on optimal regularity and regularizing effects of point mass coupling for a heat-wave system}

\author{Boris Muha\thanks{ Faculty of Science,
Department of Mathematics,University of Zagreb}}
\date{}

\maketitle

\begin{abstract}
We consider a coupled $1D$ heat-wave system which serves as a simplified fluid-structure interaction problem. The system is coupled in two different ways: the first, when the interface does not have mass and the second, when the interface does have mass. We prove an optimal regularity result in Sobolev spaces for both cases. The main idea behind the proof is to reduce the coupled problem to a single nonlocal equation on the interface by using Neummann to Diriclet operator. Furthermore, we show that point mass coupling regularizes the problem and quantify this regularization in the sense of Sobolev spaces.
\end{abstract}
\section{Introduction}

In this note we analyze a coupled system consisting of the linear wave equation and the linear heat equation coupled through a common interface. This system can be viewed as a simplified fluid-structure interaction problem \cite{ZhangZuazuaCRAS1,ZhangZuazuaCRAS2,ZhangZuazuaJDE04,ZhangZuazuaARMA07}. Fluid-structure interaction (FSI) problems naturally arise in many applications and have been extensively studied from both analytical and numerical point of view (see e.g. \cite{FSIforBio,GaldiHandbook,BorSun} and references within). Despite recent progress, the development of a comprehensive well-posedness and regularity theory for FSI problems still remains a challenge. One of the main difficulties in analysis of FSI problems is hyperbolic-parabolic coupling and corresponding mismatch in regularity of solutions. The purpose of this note is to analyze this mismatch on the simplified problem and to answer the following two questions:
\begin{enumerate}
\item What is optimal regularity for the considered system in the following sense: what is the minimal regularity for the wave component that allows the heat component to develop full parabolic regularity?
\item What is the answer to the first question in the case when the interface has a mass, i.e. when coupling is realized through point mass? Does the coupling through point mass provide additional regularity to the problem?
\end{enumerate} 
We believe that answers to these questions for the simplified problem will give us better understanding of more complex and realistic FSI models. 

Let us now briefly describe the main results of this paper. Let $u_0$ be the initial data for the heat equation and let $(v_0,v_1)$ be the initial data for the wave equation. We prove that:
\begin{enumerate}
\item Optimal regularity is obtained for $(u_0,v_0,v_1)\in H^{2s+1/2}\times H^{s+1}\times H^s$, $s\geq 0$. Then the solution of the heat component $u$ satisfies $u\in L^2(H^{2s+3/2})$ and the solution of the wave component $v$ satisfies $v\in C(H^{s+1})$, where $L^2(H^s)$ is abbreviation for $L^2(0,T;H^s(\Omega))$. Notice that the obtained function spaces are non-symmetric and are not connected to the energy of the problem (on neither level). This regularity result is optimal in the following somewhat non-standard sense. If one takes more regular initial data for the wave equation, i.e. $(v_0,v_1)\in H^{r+1}\times H^r$, $r>s$, the regularity of the solution $u$ would not increase, i.e. we would still have $u\in L^2(H^{2s+3/2})$. On the other hand, if one takes the wave initial data which are less regular, i.e. $(v_0,v_1)\in H^{r+1}\times H^r$, $r<s$, then the parabolic part of the equation will not develop the full parabolic regularity, i.e. $u\notin  L^2(H^{2s+3/2})$. This will be made precise by the regularity theorem for the heat-wave system (see Theorem~\ref{existenceThm1}).
\item If the coupling is done through point mass, the system gains $1/2$ of the derivative in a sense that initial data $(u_0,v_0,v_1)\in H^{2s+1/2}\times H^{s+1/2}\times H^{s-1/2}$ ( i.e. with $1/2$ derivative less in the wave component) produce the solution with the same regularity of the heat component as in the case without point mass, i.e. $u\in L^2(H^{2s+3/2})$. Naturally, the wave component in this case follows the regularity of the initial data, i.e. $v\in C(H^{s+1/2})$. This regularization effect of the interface with mass was noticed in \cite{SunBorMulti} where the authors considered a more realistic moving boundary fluid-multi-layered structure problem motivated by blood flow applications (for a similar effect in a different context see \cite{HansenZuazua}). However, in this work we quantify this regularization and give an explicit formula that explains the mechanism behind this regularization.
\end{enumerate}

\subsection{Brief literature review}
The same simplified FSI model as in this note was analyzed in \cite{ZhangZuazuaCRAS1,ZhangZuazuaCRAS2,ZhangZuazuaJDE04,ZhangZuazuaARMA07} where the authors addressed  boundary control, observability, stabilization and long time behavior of the solution. Rational decay rates for this model have also been studied in \cite{AvalosTriggiani2,Duyckaerts}.

In the context of strong regular solutions for FSI problems where both the fluid and the solid occupy a domain with the same spatial dimension (i.e. an elastic body is not described with some lower dimensional model), the following results have been obtained. 
A linear FSI problem on a fixed domain where $2D$ or $3D$ Stokes equations are coupled with the equations of $2D$ or $3D$ linear elasticity was studied in \cite{Gunzburger}. The existence and uniqueness of a strong solution was proved with initial data $(u_0,v_0,v_1)\in H^2\times H^2\times H^2$. A similar problem was considered in \cite{AvalosLasieckaTriggiani}, where the existence and uniqueness of solution $(u,v)\in L^2(H^2)\times L^{\infty}(H^2)$ was obtained with initial data $(u_0,v_0,v_1)\in H^2\times H^2\times H^1$. The authors noted that additional regularity for the initial structure displacement is needed in order to take advantage of the parabolic regularity for the fluid component (Remark 1.2 and Theorem 2.1). An analogous result for the nonlinear FSI problem defined on the fixed domain was proved in \cite{BarGruLasTuff}. 

D.~Coutand and S.~Shkoller proved the existence, locally in time, of a unique, regular solution for
a moving boundary FSI problem between a viscous, incompressible fluid in $3D$ and a $3D$ structure, immersed in the fluid,
where the structure was modeled by the equations of linear \cite{CSS1}, or quasi-linear \cite{CSS2} elasticity. In \cite{CSS1} initial data have the following regularity $(u_0,v_0,v_1)\in H^5\times H^3\times H^2$, while the solution satisfies $(u,v)\in L^2(H^3)\times C(H^3)$. Kukavica and Tuffaha \cite{Kuk,Kuk2} considered a similar problem where the structure was modeled by a linear wave equation. In \cite{Kuk} they proved existence, locally in time, of solution $(u,v)\in L^2(H^3)\times C^0(H^{11/4-\varepsilon})$, $\varepsilon>0$, with initial data $(u_0,v_0,v_1)\in H^3\times H^3\times H^2$, while in \cite{Kuk2}
initial data $(u_0,v_0,v_1)\in H^3\times (H^{5/2+r}\times H^{3/2+r})$ yield solution $(u,v)\in L^{\infty}(H^{5/2+r})\times C(H^{5/2+r})$, $r\in (0,(\sqrt{2}-1)/2)$. Furthermore, in \cite{IgnatovaKukavica}
the existence of solution $(u,v)\in L^{\infty}(H^3)\times C(H^3)$ was established with initial data $(u_0,v_0,v_1)\in H^4\times H^3\times H^2$. Recently, a similar problem was studied in \cite{raymond2013fluid} where the authors proved the existence of a unique solution $(u,v)\in L^2(H^{2+l})\times C(H^{7/4+l/2})$ with initial data $(u_0,v_0,v_1)\in H^{1+l}\times H^{3/2+l+\beta}\times H^{1/2+l+\beta}$, where $l\in (1/2,1)$, $\beta>0$.

A nonlinear, unsteady, moving boundary, fluid-structure interaction (FSI) problem in which the structure is composed of two layers: a thick layer, and a thin layer which serves as a fluid-structure interface with mass was studied in \cite{SunBorMulti,SunBorHyp2012} where the existence of a weak solution was proved. The authors noted that the presence of a thin fluid-structure interface with mass regularizes solutions of the coupled problem. These observations were numerically confirmed in \cite{multi-layered}. This is reminiscent of the result from \cite{HansenZuazua,KochZauzua} where two linear wave equations were coupled via elastic interface with mass and the well-posedness result was proved by taking advantage of the regularizing effects of the elastic interface.

We would like to emphasize that in most of the references cited in this short overview, the considered models are much more complicated and realistic than the model considered in this note.  Therefore, it is not clear that the presented optimal regularity result can also be obtained in these cases. Moreover, the statements of the cited results are slightly adjusted  to be comparable to the simplified $1D$ case. However, we believe that the presented analysis will provide better understanding of asymmetric regularity for the parabolic-hyperbolic systems and of regularization by point mass coupling (or coupling via elastic interface in a more realistic case).

\subsection{Notation}
In this paper we mostly use the standard notations. For example, we denote by $H^s(a,b)$ the Sobolev space of order $s$ on $(a,b)\subset\R$, $s\in\R$, and $H^s_0(a,b)$ is the closure of ${\cal D}(a,b)$ in $H^s(a,b)$ (see e. g. \cite{LionsMagenes}). Moreover, we introduce the following notation:
$$
H_{00}^s(0,T)=\overline{{\mathcal D}((0,T])}^{H^s(0,T)}=\{f\in H^s(0,T):f^{(j)}(0)=0,\; 0\leq j<s-\frac{1}{2}\}.
$$
Furthermore, we define the function spaces appropriate for analysis of parabolic problems (see e.g. \cite{LionsMagenes2}).
\begin{equation}\label{ParabolicSpace}
H^{s,2s}((0,T)\times (-1,0))=L^2(0,T;H^{2s}(-1,0))\cap H^s(0,T;L^2(-1,0)),\; s\geq 0.
\end{equation}
Finally, let us define the hyperbolic solution spaces as follows:
\begin{equation}
\begin{array}{c}
V^{s}((0,T)\times (0,1))=\{v\in C([0,T];H^s(0,1)):
\\ \\
\partial_t^k v\in C([0,T];H^{s-k}(0,1)),\; k=1,\dots, \lfloor s \rfloor \},\; s\geq 0.
\end{array}
\label{HyperbolicSpace}
\end{equation}
\section{Problem description}
We consider the following coupled problems of parabolic-hyperbolic type which can be viewed as a simplified fluid structure problem and a fluid-composite structure problem, respectively.
\begin{problem}{\rm (plain heat-wave coupling)}
\\
Find $(u,v)$ such that
\begin{equation}
\left \{
\begin{array}{lcr}
\partial_t u=\partial^2_x u,&\quad {\rm in}\quad  &(0,T)\times (-1,0), \\
\partial^2_t v=\partial^2_x v,&\quad {\rm in}\quad  &(0,T)\times (0,1),
\end{array}
\right .
\label{HeatWave}
\end{equation}
\begin{equation}
\left \{
\begin{array}{lr}
u(t,0)=\partial_t v(t,0),&\quad t\in(0,T),\\
\partial_x u(t,0)=\partial_x v(t,0),&\quad t\in (0,T),
\end{array}
\right .\
\label{HeatWaveCoupling}
\end{equation}
\begin{equation}
\left \{
\begin{array}{lr}
u(t,-1)=v(t,1)=0,&\quad t\in (0,T),\\
u(0,x)=u_0(x),&\quad x\in (-1,0),\\
v(0,x)=v_0(x),\; \partial_t v(0,x)=v_1(x),&\quad x\in (0,1).

\end{array}
\right .
\label{IBData}
\end{equation}
\label{HW}
\end{problem}
Coupling conditions~\eqref{HeatWaveCoupling} can be viewed as continuity of velocity~\eqref{HeatWaveCoupling}$_1$ (kinematic coupling condition) and continuity of normal stresses (dynamic coupling condition) for the simplified FSI model (see e.g. \cite{ZhangZuazuaARMA07}).

Let $h(t)$ denote the displacement of the point mass which serves as the heat-wave interface.
\begin{problem}{\rm (coupling through point mass)}
\\
Find $(u,v,h)$ such that
\begin{equation}
\left \{
\begin{array}{lcr}
\partial_t u=\partial^2_x u,&\quad {\rm in}\quad  &(0,T)\times (-1,0), \\
\partial^2_t v=\partial^2_x v,&\quad {\rm in}\quad  &(0,T)\times (0,1),
\end{array}
\right .
\label{Mass}
\end{equation}
\begin{equation}
\left \{
\begin{array}{lr}
u(t,0)=h'(t)=\partial_t v(t,0),&\quad t\in(0,T),\\
h''(t)=\partial_x v(t,0)-\partial_x u(t,0),&\quad t\in (0,T),
\end{array}
\right .\
\label{MassCoupling}
\end{equation}
\begin{equation}
\left \{
\begin{array}{lr}
u(t,-1)=v(t,1)=0,&\quad t\in (0,T),\\
u(0,x)=u_0(x),&\quad x\in (-1,0),\\
h(0)=0,\; h'(0)=0, \\
v(0,x)=v_0(x),\; \partial_t v(0,x)=v_1(x),&\quad x\in (0,1).

\end{array}
\right .
\label{MassIBData}
\end{equation}
\label{HWMass}
\end{problem}
Notice that dynamic coupling condition~\eqref{MassCoupling}$_2$ is exactly the second Newton's Law of motion which states that the point mass acceleration is balanced by difference of normal stresses from the wave and the heat equations. Problem 2 can be viewed as a simplified version of the FSI problem considered in \cite{multi-layered,SunBorMulti}

In order to prove regularity results for the considered problems, one has to assume that initial data satisfy certain compatibility conditions. However, to avoid technical complications and to make the text more accessible to the reader, we have chosen the simplest kind of compatibility conditions rather that pursuing full generality.  Therefore, we will assume that the initial data are in $H_0^s$ spaces and have chosen $h(0)=h'(0)=0$ as the initial conditions for $h$. 

The first step is to reformulate Problem~\ref{HW} and Problem~\ref{HWMass} in terms of
trace function $g$ defined on $(0,T)$. 
The main tool will be Neumann to Dirichlet operator for the heat equation. Since we also use D'Alembert formula, time $T$ will depend on a slope of the characteristics, in particular in the considered case we assume $T=1$. However, our argument can be iterated and the results can be extended to arbitrary $T$ (including $T=\infty$), see Remark \ref{Global}. Therefore we will continue to write $T$ instead of $1$, also to allow the reader to keep track of the different integration (w.r.t. space or time variable) more easily.
\begin{definition}[Neumann to Dirichlet operator]
Let $g:[0,T]\to\R$ and $u_0:[-1,0]\to\R$. Furthermore, let $Sg$ be a solution of the following
initial boundary value problem:
\begin{equation}\label{DefL1}
\begin{array}{c}
\partial_t (Sg)=\partial^2_x (Sg),\quad {\rm in }\quad (0,T)\times (-1,0),\\
(Sg)(.,-1)=0,\; (Sg)(0,.)=u_0,\; \partial_x (Sg)(.,0)=g.
\end{array}
\end{equation}
We define Neumann to Dirichlet operator $L$ with the following formula:
\begin{equation}\label{ND}
(L_{u_0}g)(t):=(Sg)(t,0),\quad t\in (0,T),
\end{equation}
where equality is taken in the trace sense.
\end{definition}
\begin{remark}
Poincar\' e-Steklov Neumman to Dirichlet operator for Laplace's equation was used in \cite{AddedMassCNG} in the analysis of the so-called added-mass effect and its connection to stability issues for the numerical schemes for the FSI problems involving the lower dimensional elastic models.
\end{remark}
Since problem~\eqref{DefL1} is linear and the trace operator is also linear, we can decompose $L_{u_0}$ in the following way:
\begin{equation}\label{Lu0}
L_{u_0}g=L_{u_0}{\bf 0}+L_0g, 
\end{equation}

\begin{proposition}\label{LIso}
Let $u_0\in H^r_0(-1,0)$, $r\geq 1$, and $g\in H_{00}^s(0,T)$, $s\geq 1/4$. Then operator $L_{u_0}$ is well-defined. Furthermore, the following statements hold:
\begin{enumerate}
\item Let $r\neq 2n+2$ and $r\neq n+1/2$, $n\in\N_0$, then we have $L_{u_0}{\bf 0}\in H^{r/2+1/4}(0,T)$.
\item Let $s\neq n+3/4$ and $s\neq n/2$, $n\in\N_0$, then 
$L_0:H_{00}^s(0,T)\to H_{00}^{s+1/2}(0,T)$ is an isomorphism.
\end{enumerate}
\end{proposition}
\begin{proof}
The fact that operator $L_{u_0}$ is well-defined is a direct consequence of Theorem 4.4.3. (see also Remark 4.4.1.) from \cite{LionsMagenes2}, which states that problem~\eqref{DefL1} has a unique solution $Sg\in H^{1,2}((0,T)\times (-1,0))$ for $u_0\in H^1_0(-1,0)$ and $g\in H^{1/4}(0,T)$. Therefore, the trace operator in~\eqref{ND} is well-defined (see e.g. Theorem 4.2.1. in \cite{LionsMagenes2}).

To prove statements 1 and 2 we need a regularity result for problem~\eqref{DefL1}, namely Theorem 4.6.2. from \cite{LionsMagenes2}. Since our initial and boundary data satisfy $u_0\in H^r_0(-1,0)$ and $g\in H^s_{00}(0,T)$, respectively, the compatibility conditions in point $(0,0)$ are satisfied. Furthermore, restrictions on parameters $r$ and $s$ are necessary to satisfy the assumptions of Theorem 4.6.2. from \cite{LionsMagenes2}.

To prove statement 1, we apply Theorem 4.6.2. to problem~\eqref{DefL1} with zero Neumann boundary data to prove that the solution belongs to space $H^{(r+1)/2,r+1}((0,T)\times (-1,0))$. Now, the statement follows by direct application of the trace theorem for $H^{s,2s}$ spaces and the fact that $u_0\in H^r_0(-1,0)$.

Statement 2 can be proven with analogous reasoning. Namely, now we consider problem~\eqref{DefL1} with zero initial data, and apply Theorem 4.6.2. (see also Remark 4.6.3. in \cite{LionsMagenes2}) and the trace theorem.
\end{proof}

We use D'Alembert formula in the wave subdomain to rewrite the full coupled problem as a single nonlocal equation on the interface.
Let us denote with $I$ the area where the solution of the wave equation is only influenced by the interface
$$
I=\{(t,x)\in (0,1)\times (0,\frac 1 2):|2t-1|\leq 1-2x\},
$$
and by $II$ we denote the area where the solution of the wave equation is only influenced by the initial data (Figure \ref{fig})
$$
II=\{(t,x)\in (0,\frac 1 2)\times (0,1): |2x-1|\leq 1-2t\}.
$$
\begin{figure}[ht]
\centering{
\includegraphics[scale=1]{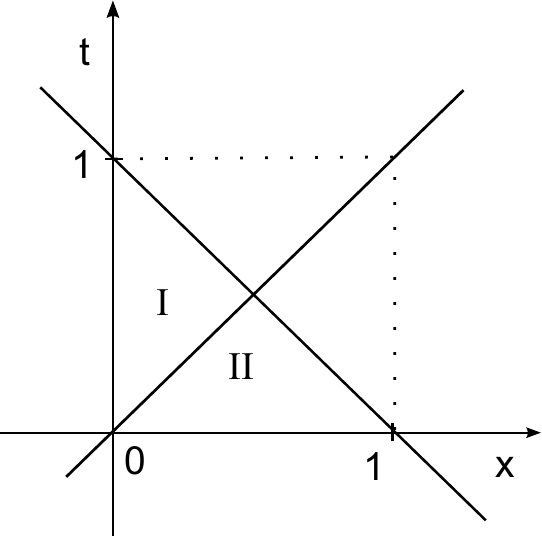}
}
\caption{Characteristics}
\label{fig}
\end{figure}
Solution $v$ on $II$ is given by d'Alembert's formula:
\begin{equation}\label{DA2}
v(t,x)=\frac 1 2 (v_0(x-t)+v_0(x+t))+\frac 1 2\int_{x-t}^{x+t}v_1(s)ds,\quad (t,x)\in II.
\end{equation}
Similarly as in \cite{HansenZuazua}, we exploit the fact that $1D$ wave equation is symmetric in $t$ and $x$ variables and therefore $v$ on $I$ is also given by d'Alembert's formula when the roles of $t$ and $x$ are switched, and we consider region $I$ as being influenced only by the boundary data on $(0,1)\times\{0\}$. Therefore, we have the following formula for solution $v$:
\begin{equation}\label{DA1}
v(t,x)=\frac 1 2(h(t-x)+h(t+x))+\frac 1 2\int_{t-x}^{t+x}\partial_x v(s,0)ds,\quad (t,x)\in I,
\end{equation}
where $h(t)=v(t,0)$. Now, assuming the continuity of $v$ on the ray given by formula $t=x$ and using formulas \eqref{DA1} and \eqref{DA2} at point $(\frac t 2,\frac t 2)$ we get:
$$
h(t)+\int_0^t\partial_x v(s,0)ds=v_0(t)+\int_{0}^tv_1(s)ds.
$$
By differentiating this equality we get
\begin{equation}\label{InterfaceEq}
c(s)+\partial_x v(s,0)=v_0'(s)+v_1(s),
\end{equation}
where $c(t)=h'(t)$ is the interface velocity.
Now we are in position to reformulate Problems \ref{HW} and \ref{HWMass} in terms of the Neumann to Dirichet operator $L_{u_0}$.
\begin{proposition} Let $L_{u_0}$ be Neumann to Diriclet operator defined by~\eqref{ND},
let $g=\partial_x u(.,0)$ and $T=1$. Then for the wave equation initial data $(v_0,v_1)$ the following statements hold:  
\begin{enumerate}
\item
Problem \ref{HW} is formally equivalent to the following nonlocal problem:
\\
Find $g$ such that
\begin{equation}\label{HWequiv}
(I+L_{u_0})g=v_0'+v_1\;{\rm in}\; (0,T).
\end{equation}
\item
Problem \ref{HWMass} is formally equivalent to the following nonlocal problem:
\\
Find $g$ such that
\begin{equation}\label{HWMassequiv}
\begin{array}{c}
(I+L_{u_0})g+(L_{u_0}g)'=v_0'+v_1\;{\rm in}\; (0,T),
\\
(L_{u_0}g)(0)=0.
\end{array}
\end{equation}
\end{enumerate}
\label{equiv}
\end{proposition}
\begin{proof}
First we notice that from kinematic coupling condition~\eqref{HeatWaveCoupling}$_1$ (or~\eqref{MassCoupling}$_1$)
and the definition of operator $L_{u_0}$ we have
$$
c=h'=\partial_t v(.,0)=u(.,0)=L_{u_0}(\partial_x u(.,0))=L_{u_0}g.
$$
Now~\eqref{HWequiv} follows directly from~\eqref{InterfaceEq} and the dynamic coupling condition~\eqref{HeatWaveCoupling}$_2$, i.e.  $g=\partial_x u(.,0)=\partial_x v(.,0)$. Similarly,
~\eqref{HWMassequiv} follows directly from~\eqref{InterfaceEq} and the dynamic coupling condition~\eqref{MassCoupling}$_2$, i.e.  $\partial_x v(.,0)=\partial_x u(.,0)+h''=g+c'$. 
\end{proof}
\begin{remark}
The extra term $(L_{u_o} g)'$ in~\eqref{HWMassequiv} comes from the inertia of the point mass located at the interface between the heat and the wave equation domains. 
\end{remark}
Using the Dirichlet to Neumann formulation of problems 1 and 2 we show next that for problem 2, which contains the interface with point mass, less regularity of the initial data is required to recover the same interface regularity as in problem 1, which has no point mass at the interface. More precisely, we have the following two Propositions. 
\begin{proposition}\label{existence}
Let $u_0\in H^{r_1}_0(-1,0)$, $r_1\geq 1$, and $(v_0,v_1)\in H^{r_2+1}_0(0,1)\times H^{r_2}_0(0,1)$, $r_2\geq 1/4$, and $T=1$. Then there exists a unique solution $g\in H^{1/4}(0,T)$ to problem~\eqref{HWequiv}. Furthermore, the following statements hold
\begin{enumerate}
\item Let $2r_2+1/2\leq r_1+1$, and $r_2\neq n+3/4$ and $r_2\neq n/2$, $n\in\N_0$. Then $\displaystyle{g\in H_{00}^{r_2}(0,T)}$.
\item Let $r_1+1\leq 2r_2+1/2$, and $r_1\neq 2n+2$ and $r_1\neq n+1/2$, $n\in\N_0$ . Then $\displaystyle{g\in H_{00}^{r_1/2+1/4}(0,T)}$.
\end{enumerate}
\end{proposition}
\begin{proof}
Using equation~\eqref{Lu0}, we can rewrite~\eqref{HWequiv} in the following way:
\begin{equation}\label{TMP1}
(I+L_0)g=v_0'+v_1-L_{u_0}{\bf 0}.
\end{equation}
Notice that the right hand side of the above equation is a $H^{1/4}(0,1)$ function (see Proposition \ref{LIso}). Furthermore, we can view $L_0$ as an operator on $H^{1/4}(0,1)$, i.e. $L_0:H^{1/4}(0,1)\to H^{1/4}(0,1)$. Now, $L_0$ is a compact operator on $H^{1/4}(0,1)$ because of compactness of embedding $H^{3/4}(0,1) \hookrightarrow H^{1/4}(0,1)$. Therefore, we can use Fredholm alternative (see e.g. Theorem 6.6 in \cite{Brezis}) and to complete the proof of statement 1, it only remains to prove ker$(I+L_0)=\{\bf 0\}.$

Let us take $g\in{\rm ker}(I+L_0)$, i.e. $L_0g=-g$. By multiplying~\eqref{DefL1} by $Sg$, integrating on
$(0,t)\times(-1,0)$ and integrating by parts, we get the following equality for every $t\in (0,1]$:
$$
\frac 1 2\|Sg(t,.)\|_{L^2(-1,0)}^2=\frac 1 2\|Sg(0,.)\|_{L^2(-1,0)}^2-\|\partial_x (Sg)\|^2_{L^2((0,t)\times (-1,0))}-\|g\|_{L^2(0,t)}^2.
$$ 
Now from the initial condition $Sg(0,.)=0$, we conclude $Sg=0$ and therefore $g=Sg(.,0)=0$. This concludes the proof of the existence and the uniqueness part of the Proposition.

Let us now prove the regularity statements. Let us denote by $F$ the right-hand side of equation~\eqref{TMP1}. The condition on $r_1$ and $r_2$ in statement 1 can be rewritten as $r_2\leq r_1/2+1/4$. Therefore, by using statement 1 from Proposition \ref{LIso}, we get $F\in H^{r_2}_{00}(0,1)$. Now, by using the same reasoning as in the proof of the existence part, and the fact that $L_0:H_{00}^{r_2}(0,1)\to H_{00}^{r_2+1/2}(0,1)$ (see Proposition \ref{LIso}, statement 2), we conclude $g\in H_{00}^{r_2}(0,1)$.

The proof of statement 2 is analogous. Namely, in this case we also use Proposition \ref{LIso} to conclude $F\in H_{00}^{r_1/2+1/4}(0,1)$ and use the fact that $L_0:H_{00}^{r_1/2+1/4}(0,1)\to H_{00}^{r_1/2+3/4}(0,1)$.
\end{proof}
Notice that in the case $2r_2+1/2\leq r_1+1$ the regularity of $g$ is determined by $r_2$, i.e. by the regularity of the wave equation initial data, while in the other case the regularity of $g$ is determined by $r_1$, i.e. the heat equation initial data.

Now let us prove an analogous result for the case of coupling through point mass.

\begin{proposition}\label{existenceMass}
Let $u_0\in H^{r_1}_0(-1,0)$, $r_1\geq 1$, and $(v_0,v_1)\in H^{r_2+1}_0(0,1)\times H^{r_2}_0(0,1)$, $r_2\geq -1/4$. Then there exists a unique solution $g\in H^{1/4}(0,T)$ to problem~\eqref{HWMassequiv}. Furthermore, the following statements hold
\begin{enumerate}
\item Let $2r_2+3/2\leq r_1$, and $r_2\neq n+3/4$ and $r_2\neq n/2$, $n\in\N_0$. Then $\displaystyle{g\in H_{00}^{r_2+1/2}(0,T)}$ and $\displaystyle{L_{u_0}g\in H_{00}^{r_2+1}(0,T)}$.
\item Let $r_1\leq 2r_2+3/2\leq r_1+1$, and $r_1\neq 2n+2$ and $r_1\neq n+1/2$, $n\in\N_0$ . Then $\displaystyle{g\in H_{00}^{r_1/2-1/4}(0,T)}$ and $\displaystyle{L_{u_0}g\in H_{00}^{r_2+1}(0,T)}$.
\item Let $r_1+1\leq 2r_2+3/2$, and $r_1\neq 2n+2$ and $r_1\neq n+1/2$, $n\in\N_0$ . Then $\displaystyle{g\in H_{00}^{r_1/2-1/4}(0,T)}$ and $\displaystyle{L_{u_0}g\in H_{00}^{r_1+3/4}(0,T)}$.
\end{enumerate}
\end{proposition}
\begin{proof}
We denote the right-hand side of~\eqref{HWMassequiv} by $f=v_0'+v_1$. We can formally solve Cauchy problem~\eqref{HWMassequiv} for unknown $L_{u_0}g$ and get 
\begin{equation}\label{FomulaLu0}
(L_{u_0}g)(t)=\int_0^te^{s-t}(f(s)-g(s))ds.
\end{equation}
By using~\eqref{Lu0} we get
$$
L_0g+\int_0^te^{s-t}g(s)ds=\int_0^te^{s-t}f(s)ds-L_{u_0}{\bf 0}.
$$
We apply $L_0^{-1}$ to obtain the following formulation, which is formally equivalent to~\eqref{HWMassequiv}:
\begin{equation}\label{existenceTmp1}
(I+W)g=L_0^{-1}\big ( \int_0^te^{s-t}f(s)ds-L_{u_0}{\bf 0}\big ),
\end{equation}
where $(Wg)(t)=L_0^{-1}\big (\int_0^te^{s-t}g(s)ds\big )$. We again use Fredholm alternative to prove the existence result. Namely, $W$ is a well defined operator on $H^{1/4}(0,1)$. Furthermore, it is compact since Im$(W)\subset H^{3/4-\varepsilon}(0,1)$, $\varepsilon>0$ (by integration we gain one derivative and by applying Dirichlet to Neumann operator $L_0^{-1}$ we lose half of a derivative, see Proposition \ref{LIso}, statement 2. Therefore it only remains to prove ker$(I+W)=\{\bf 0\}.$

Let $g\in{\rm ker}(I+W)$. Then
$$
(L_0g)(t)=-\int_0^te^{s-t}g(s)ds.
$$
Using a calculation analogous to the one in the proof of Proposition~\ref{existence} we get the following equality for every $t\in (0,1]$:
\begin{equation}\label{existenceTmp2}
\begin{array}{c}
\displaystyle{\frac 1 2\|Sg(t,.)\|_{L^2(-1,0)}^2=\frac 1 2\|Sg(0,.)\|_{L^2(-1,0)}^2-\|\partial_x (Sg)\|^2_{L^2((0,t)\times(-1,0))}}
\\
\displaystyle{-\int_0^tg(\tau)\int_0^\tau e^{s-t}g(s)dsd\tau.}
\end{array}
\end{equation}
To show that the right-hand side of~\eqref{existenceTmp2} is negative, we define $G(\tau)=\int_0^\tau e^{s-t}g(s)ds$. Straightforward calculation yields $g=G'+G$. Therefore, we have
$$
-\int_0^tg(\tau)\int_0^\tau e^{s-t}g(s)dsdt=-\int_0^t(G'+G)G=-\frac 1 2 G^2(t)-\int_0^tG^2\leq 0.
$$
Now, from~\eqref{existenceTmp2} we deduce $g=0$, i.e.  ker$(I+W)=\{\bf 0\}.$

Hence, we proved the existence of unique $g\in H^{1/4}(0,1)$ satisfying~\eqref{existenceTmp1}, where integral $\int_0^te^{s-t}f(s)$ is understood in a dual $H^{-1/4}(0,1)$ sense. By applying $L_0$ on~\eqref{existenceTmp1} and by differentiating the resulting equation, we show that $g$ is a unique solution to the problem~\eqref{HWMassequiv}, where equality~\eqref{HWMassequiv} is understood in the distributional sense on $(0,1)$. Uniqueness follows from the facts that $L_{u_0}g\in H^{3/4}(0,1)$ and $(L_{u_0}g)(0)=0$.

To prove the regularity part we first notice that from Proposition \ref{LIso}, statement 2, we can conclude that
$W:H_{00}^{s+1/2}(0,1)\to H_{00}^{s+1}(0,T)$ is an isomorphism,  $s\neq n+3/4$ and $s\neq n/2$, $n\in\N_0$. Let us denote the right-hand side of~\eqref{existenceTmp1} by $F$. Analogously to the proof of Proposition~\ref{existence}, the regularity of $g$ follows from the regularity of $F$. Term $L_0^{-1}\int_0^te^{s-t}f(s)ds$ belong to $H^{r_2+1/2}$ space, $r_2\neq n+3/4$ and $r_2\neq n/2$, $n\in\N_0$. Therefore, using Proposition~\ref{LIso} we conclude that
$$
F\in\left \{\begin{array}{c}H^{r_2+1/2}_{00}(0,1),\; r_2+1/2\leq r_1/2-1/4,
\\
H^{r_1/2-1/4}_{00}(0,1),\; r_2+1/2\geq r_1/2-1/4.\end{array} \right.
$$
This concludes the proof of the regularity of $g$. It only remains to prove the statements about the regularity of $L_{u_0}g$. However, $L_{u_0}g$ is given by formula~\eqref{FomulaLu0} and therefore the regularity of $L_{u_0}g$ follows directly from regularity of $f$ and $g$. 
\end{proof}
Similarly as in Proposition~\ref{existence}, the regularity of $g$ in the case covered by statement 1 is determined by $r_2$, while in the cases covered by statements 2 and 3 is determined by $r_1$. However, the regularity of $L_{u_0}g$ (which represents the interface velocity) is determined by $r_2$ in statements 1 and 2 and by $r_1$ in statement 3.

\begin{remark}
Notice that for the proofs of Proposition~\ref{existence} and~\ref{existenceMass} conditions $v_0(0)=0$ and $h(0)=0$ are not necessary and both Propositions are valid without these conditions. However, we opted to leave assumption $v_0\in H^{r_2+1}_0(0,1)$ in the statements of the Propositions for the notational simplicity.
\end{remark}

Let us conclude this section by proving the existence and the regularity theorem for original coupled Problems \ref{HW} and \ref{HWMass}. 
\begin{theorem}
Let $u_0\in H^{r_1}_0(-1,0)$, $r_1\geq 1$, $(v_0,v_1)\in H^{r_2+1}_0(0,1)\times H^{r_2}_0(0,1)$, $r_2\geq 1/4$, $T=1$, and let $g$ be a solution to problem~\eqref{HWequiv} given by Proposition~\ref{existence}. Furthermore, let $c=L_{u_0}g$, $u=Sg$, where $L_{u_0}$ and $S$ are defined by \eqref{ND} and \eqref{DefL1}, respectively; and $h(t)=\int_0^tc(t)dt$. Finally, let $v$ be a solution to the following initial boundary value problem for the wave equation:
\begin{equation}\label{Wave}
\begin{array}{c}
\partial_t^2v=\partial^2_x v,\quad {\rm in}\quad (0,T)\times (0,1),\\
v(t,0)=h(t),\; v(t,1)=0,\quad t\in (0,T),\\
v(0,x)=v_0(x),\; \partial_t v(0,x)=v_1(x),\quad x\in (0,1).
\end{array}
\end{equation}
Then $(u,v)$ is a unique solution to Problem 1. Furthermore, the following statements hold: 
\begin{enumerate}
\item Let $2r_2+1/2< r_1$, and $r_2\neq n+3/4$ and $r_2\neq n/2$, $n\in\N_0$. Then 
$$\displaystyle{(u,v)\in H^{r_2+3/4,2(r_2+3/4)}((0,T)\times(-1,0))\times V^{r_2+1}((0,T)\times (0,1))}.$$
\item Let $r_1\leq 2r_2+1/2\leq r_1+1$, and $r_2\neq n+3/4$ and $r_2\neq n/2$, $n\in\N_0$. Then 
$$\displaystyle{(u,v)\in H^{(r_1+1)/2,r_1+1}((0,T)\times(-1,0))\times V^{r_2+1}((0,T)\times (0,1))}.$$
\item Let $r_1+1\leq 2r_2+1/2\leq r_1+2$, and $r_1\neq 2n+2$ and $r_1\neq n+1/2$, $n\in\N_0$ . Then 
$$\displaystyle{(u,v)\in H^{(r_1+1)/2,r_1+1}((0,T)\times(-1,0))\times V^{r_2+1}((0,T)\times (0,1))}.$$
\item Let $r_1+2< 2r_2+1/2$, and $r_1\neq 2n+2$ and $r_1\neq n+1/2$, $n\in\N_0$ . Then 
$$\displaystyle{(u,v)\in H^{(r_1+1)/2,r_1+1}((0,T)\times(-1,0))\times V^{r_1/2+7/4}((0,T)\times (0,1))}.$$
\end{enumerate}
\label{existenceThm1}
\end{theorem} 
\begin{proof}
The proof follows from Propositions \ref{equiv} and \ref{existence}, and the regularity results for the heat an the wave equation. Let us first prove that $(u,v)$ is a unique solution to Problem 1. Namely, from $g\in H^{1/4}(0,T)$ we get $u\in H^{1,2}((0,T)\times (-1,0))$ and $v\in V^{5/4}((0,T)\times (0,1))$. With stated regularity we can rigorously justify all the steps that lead to the formal equivalence of Problem \ref{HW} and~\eqref{HWequiv}. Moreover, coupling conditions on the interface \eqref{HeatWaveCoupling} for Problem \ref{HWMass} are satisfied in the trace sense, where one has to use the so-called hidden regularity theorem for the wave equation to justify trace $\partial_x v(.,0)$. More precisely, we have $\partial_x v(.,0)\in L^2(0,1)$ (see e.g. \cite{LionsHiddenRegularity}, Theorem 2.1.). 

To prove statements 1-4, we use Theorem 4.6.2. from \cite{LionsMagenes2} for the regularity of the heat equation, and  Theorem 2.5. and Remark 2.10. from \cite{LionsHiddenRegularity} for the regularity for the wave equation. Let us prove statement 1.

From Proposition \ref{existence} we have $g\in H^{r_2}_{00}(0,T)$. Moreover, from Proposition \ref{LIso} and~\eqref{Lu0} we have $c=L_{u_0}g\in H^{r_2+1/2}_{00}(0,T)$ and consequently $h\in H^{r_2+3/2}_{00}(0,T)$. Now, the statement follows from the regularity results for the heat and weave equations.

Statements 2-4 follows analogously by using Propositions \ref{LIso} and \ref{existence} together with the regularity theorems for the heat and the wave equations.
\end{proof}
In the case covered by statement 1 ($2r_2+1/2<r_1$) the regularity of heat component $u$ is not optimal w.r.t. the initial data. Namely, the full parabolic regularity with initial data $u_0\in H^{r_1}$ would yield $u\in H^{(r_1+1)/2,r_1+1}$. The reason for this loss of regularity is the influence of the less regular wave initial data via the coupling conditions. On the other hand, the wave component has optimal regularity, i.e. $v\in V^{r_2+1}$.

In the cases covered by statements 2 and 3 ($r_1\leq 2r_2+1/2\leq r_1+2$) both components, the wave and the heat, have optimal regularity w.r.t. the initial data. 

Finally, in the case covered by statement 4 ($r_1+2<2r_2+1/2$) the heat component has optimal regularity, while the wave component does not. Again, the reason for the loss of regularity is the influence of less regular heat initial data via the coupling conditions.

\begin{theorem}
Let $u_0\in H^{r_1}_0(-1,0)$, $r_1\geq 1$, $(v_0,v_1)\in H^{r_2+1}_0(0,1)\times H^{r_2}_0(0,1)$, $r_2\geq 0$, $T=1$, and let $g$ be a solution to problem~\eqref{HWMassequiv} given by Proposition~\ref{existenceMass}. Furthermore, let $c=L_{u_0}g$, $u=Sg$, where $L_{u_0}$ and $S$ are defined by \eqref{ND} and \eqref{DefL1}, respectively; and $h(t)=\int_0^tc(t)dt$. Finally, let $v$ be a solution to problem~\eqref{Wave}.

Then $(u,v,h)$ is a unique solution to Problem 2. Furthermore, the following statements hold: 
\begin{enumerate}
\item Let $2r_2+3/2< r_1$, and $r_2\neq n+3/4$ and $r_2\neq n/2$, $n\in\N_0$. Then 
$$\displaystyle{(u,v,h)\in H^{r_2+5/4,2(r_2+5/4)}((0,T)\times(-1,0))\times V^{r_2+1}((0,T)\times (0,1))}\times H^{r_2+2}_{00}(0,T).$$
\item Let $r_1\leq 2r_2+3/2\leq r_1+1$, and $r_2\neq n+3/4$ and $r_2\neq n/2$, $n\in\N_0$. Then 
$$\displaystyle{(u,v,h)\in H^{(r_1+1)/2,r_1+1}((0,T)\times(-1,0))\times V^{r_2+1}((0,T)\times (0,1))\times H^{r_2+2}_{00}(0,T)}.$$
\item Let $r_1+1\leq 2r_2+3/2\leq r_1+3$, and $r_1\neq 2n+2$ and $r_1\neq n+1/2$, $n\in\N_0$. Then 
$$\displaystyle{(u,v,h)\in H^{(r_1+1)/2,r_1+1}((0,T)\times(-1,0))\times V^{r_2+1}((0,T)\times (0,1))\times H^{r_1/2+7/4}_{00}(0,T)}.$$
\item Let $r_1+3< 2r_2+3/2$, and $r_1\neq 2n+2$ and $r_1\neq n+1/2$, $n\in\N_0$ . Then 
$$\displaystyle{(u,v,h)\in H^{(r_1+1)/2,r_1+1}((0,T)\times(-1,0))\times V^{r_1/2+7/4}((0,T)\times (0,1))\times H^{r_1/2+7/4}_{00}(0,T)}.$$
\end{enumerate}
\label{existenceThm2}
\end{theorem} 
\begin{proof}
The proof is analogous to the proof of Theorem~\ref{existenceThm1}. Therefore we omit the proof and just emphasize the points where the proofs differ. The main difference is that we use Proposition~\ref{existenceMass}, instead of Proposition~\ref{existence}, for the existence and the regularity of $g$. Furthermore, Proposition~\ref{existenceMass} gives us the additional regularity of $L_{u_0}g$ which is then used in the proof. Finally, we have $h(t)=\int_0^t(L_{u_0}g)(s)ds$.
\end{proof}

Similarly as in the discussion after Theorem~\ref{existenceThm1}, one can see that in the case of statement 1 the heat component does not have optimal regularity, while the wave component does have optimal regularity. In the case of statements 2 and 3 both components have optimal regularity, while in the case of statement 4 only the wave component has optimal regularity.

\begin{remark}\label{Trace}
Notice that the minimal regularity assumptions for the initial data $(v_0,v_1)$ are higher in Theorem~\ref{existenceThm2} ($r_2\geq 0$) than in Proposition~\ref{existenceMass} ($r_2\geq -1/4$). This is due to the fact that one needs certain regularity for the solution of the wave equation to make sense of the traces needed in the coupling conditions. In formulation~\eqref{HWMassequiv} the coupling conditions are ``encoded'' in operator $L_0$ and are implicit, so one can define lower regularity solutions.
\end{remark}
\begin{remark}
Let us fix parameter $r_1$ from Theorems~\ref{existenceThm1} and \ref{existenceThm2}. In the discussions after the Theorems we concluded that both components, the heat and the wave, have optimal regularity in the cases covered in statements 2 and 3. Therefore in the case of the plain heat wave coupling there is no loss of the regularity in neither component if $r_2\in[r_1/2-1/4,r_1/2+3/4]$ (Theorem~\ref{existenceThm1}). On the other hand, in the case of the coupling through point mass there is no loss of the regularity in neither component if $r_2\in[r_1/2-3/4,r_1/2+3/4]$ (Theorem~\ref{existenceThm2}). Notice that the interval is larger in the case of the coupling through point mass, which is due to the regularization effect of the point mass coupling.
\end{remark}

\section{Optimal regularity}
In this section we answer the question posed in the Introduction and give the optimal regularity result for Problems \ref{HW} and \ref{HWMass} which is a direct consequence of Theorems~\ref{existenceThm1} and~\ref{existenceThm2}. We start by the optimal regularity result for Problems~\eqref{HWequiv} and~\eqref{HWMassequiv}.
\begin{corollary}\label{existenceReg}
Let $u_0\in H_0^{2s+1/2}(-1,0)$, $s\geq 1/4$, ${s\neq n/2,\; s\neq n+\frac 3 4}$, $n\in\N_0$, and $T=1$.
\begin{enumerate}
\item Let $(v_0,v_1)\in H_0^{s+1}(0,1)\times H^{s}_{0}(0,1)$. Then there exists a unique solution $g\in H_{00}^{s}(0,T)$ to problem~\eqref{HWequiv}.
\item Let $(v_0,v_1)\in H_0^{s+1/2}(0,1)\times H_{0}^{s-1/2}(0,1)$. Then there exists a unique solution $g\in H_{00}^{s}(0,T)$ to problem~\eqref{HWMassequiv}.
\end{enumerate}
\end{corollary}
\begin{proof}
This corollary is a direct consequence of Propositions \ref{existence} and \ref{existenceMass}.
\end{proof}
\begin{theorem}
Let $u_0\in H_0^{2s+1/2}(-1,0)$, $s\geq 1/4$, ${s\neq n/2,\; s\neq n+\frac 3 4}$, $n\in\N$ and let $g\in H_{00}^{s}(0,T)$ be given by Corollary~\ref{existenceReg}, and let $T=1$. Furthermore, let $c=L_{u_0}g$, $u=Sg$, where $L_{u_0}$ and $S$ are defined by \eqref{ND} and \eqref{DefL1}, respectively; and $h(t)=\int_0^tc(t)dt$. Finally, let $v$ be a solution to the initial boundary value problem~\eqref{Wave}.
Then the following statements hold:
\begin{enumerate}
\item If $(v_0,v_1)\in H_0^{s+1}(0,1)\times H_{0}^{s}(0,1)$, then 
$$\displaystyle{(u,v)\in H^{s+3/4,2s+3/2}((0,T)\times (-1,0))\times V^{s+1}((0,T)\times(0,1))}
$$ is a unique solution to Problem \ref{HW},
\item If $(v_0,v_1)\in H_0^{r+1/2}(0,1)\times H_{0}^{r-1/2}(0,1)$, then 
$$
(u,v,h)\in H^{r+3/4,2r+3/2}((0,T)\times (-1,0))\times V^{r+1/2}((0,T)\times(0,1))\times H^{r+3/2} (0,T)
$$ is a unique solution to Problem \ref{HWMass}, where $r=\max\{s,1/2\}$.
\end{enumerate}
\label{OptimalThm}
\end{theorem} 
\begin{proof}
This Theorem is a direct consequence of Theorems \ref{existenceThm1} and \ref{existenceThm2}. Parameter $r$ is introduced because $r\geq 1/2$ ensures that trace $\partial_x v(.,0)$ in coupling condition~\eqref{MassCoupling} is well-defined (see Remark \ref{Trace}). 
\end{proof}

\begin{remark}
The displacement of the interface is regularized because of the parabolic regularity. Namely, we have $(h,c)\in H^{s+3/2}(0,T)\times H^{s+1/2}(0,T)$, which is a gain of $1/2$ derivative w.r.t. the wave displacement and velocity $(v,\partial_t v)$. However, the wave that is reflected from the interface is not regularized due to the low regularity of $g=\partial_x u(.,0)\in H^{s}(0,T)$ in the case of Problem~\ref{HW}, or due to the low regularity of $h''\in H^{s-1/2}$ in the case of Problem~\ref{HWMass}.
\end{remark}
\begin{remark}
In the case of the plain heat wave coupling we have that the interface displacement is a $H^{s+3/2}$  function (see preceding remark) with the wave initial data $(v_0,v_1)\in H^{s+1}(0,1)\times H^{s}(0,1)$. On the other hand, in the case of coupling through point mass, we have $h\in H^{s+3/2}(0,T)$ with the initial data $(v_0,v_1)\in H^{s+1/2}(0,1)\times H^{s-1/2}(0,1)$, $s\geq 1/2$. Therefore, in the case of coupling through point mass we obtained that the interface displacement has the same regularity as the plain coupling case, but with the initial data which are less regular by a degree of $1/2$ of a derivative. Therefore, we see that the coupling through point mass regularizes the system by gaining $1/2$ of a derivative.
\end{remark}

\begin{remark}{\bf On a global in time solution}\\
Using the same techniques as in Theorems \ref{existenceThm1} and \ref{existenceThm2} one could prove the existence of a global in time solution by restarting the proof from time $t=1$ and reiterating the procedure. This would yield a global in time solution since the length of time interval on every step would be $1$. However, in order to complete the proof, one would have to work with general compatibility conditions and a non-zero
right-hand side. Since in this note we are primarily interested in optimal regularity and regularizing effects of a point mass coupling, we skip the details for technical simplicity.
\label{Global}
\end{remark}

\section{Conclusions}
In this note we study the heat-wave systems of equations which are coupled at the interface between the two respective domains in two different ways: with and without point mass at the interface. These systems can be viewed as simplified FSI models. We prove an optimal regularity theorem for both systems. The regularity theorem is optimal in the sense that we have minimal regularity assumptions for the wave initial data for which we can use the full parabolic regularity for the heat component of the solution. A further increase in the regularity of the wave initial data would not yield an increase in the regularity of the heat component of the solution. We were also interested in the regularity of the interface and the regularizing effects of a point mass coupling. Our analysis revealed the following properties of the considered coupled systems:
\begin{enumerate}
\item Even in the case of the ``plain'' heat-wave coupling, namely the case without point mass, the interface displacement is regularized by a degree of $1/2$ of the derivative w.r.t. to the displacement of the wave component of the solution. This effect is a consequence of the parabolic regularity of the heat component. However, the wave which is reflected from the interface has the same regularity as the incoming wave, so there is no regularization in the wave component.
\item The point mass coupling regularizes the problem by $1/2$ of the derivative in a sense that the wave initial data needed for the optimal regularity result have $1/2$ derivative less (in the sense of Sobolev spaces) than in the ``plain'' heat wave coupling. Moreover, the interface displacement is now regularized by a degree of one derivative w.r.t. to the displacement of the wave component of the solution. However, there is still no regularization in the wave component of the solution.
\end{enumerate}
\noindent
{\bf Acknowledgments} The author would like to thank Prof. Enrique Zuazua for pointing out the references for the heat-wave system and inspiring discussions during the author's visit to BCAM. The author would also like to thank Prof. Sun\v cica \v Cani\' c for carefully reading the manuscript and for her insightful comments about the manuscript. Further thanks are extended to Prof. Zvonimir Tutek and Prof. Igor Vel\v ci\' c for their comments about the manuscript. Furthermore, the author acknowledges support by the following grants: ESF OPTPDE - Exchange Grant 4171, the National Science Foundation grant DMS-1311709, MZOS grant number 0037-0693014-2765 and Croatian Science Foundation grant number 9477.
\bibliographystyle{plain}
\bibliography{c:/Users/Boris/Dropbox/myrefs}
\end{document}